\documentclass{article}
\usepackage{amsmath}
\usepackage{amssymb}

\newtheorem{theorem}{Theorem}

\newtheorem{corollary}[theorem]{Corollary}

\newtheorem{lemma}[theorem]{Lemma}

\newtheorem{proposition}[theorem]{Proposition}
\newtheorem{remark}[theorem]{Remark}

\newenvironment{proof}[1][Proof]{\noindent\textbf{#1.} }{\ \rule{0.5em}{0.5em}}
\input{tcilatex}

\begin{document}

\title{Local behaviour of first passage probabilities}
\author{ R. A. Doney}
\date{}
\maketitle

\begin{abstract}
Suppose that $S$ is an asymptotically stable random walk with norming
sequence $c_{n}$ and that $T_{x}$ is the time that $S$ first enters $%
(x,\infty ),$ where $x\geq 0.$ The asymptotic behaviour of $P(T_{0}=n)$ has
been described in a recent paper of Vatutin and Wachtel \cite{vw}, and here
we build on that result to give three estimates for $P(T_{x}=n),$ which hold
uniformly as $n\rightarrow \infty $ in the regions $x=o(c_{n}),$ $%
x=O(c_{n}), $ and $x/c_{n}\rightarrow \infty ,$ respectively.
\end{abstract}

\section{Introduction}

Supppose $S$ is a 1-dimensional random walk and for $x\geq 0$ let $T_{x}$ be
the first exit time of $(-\infty ,x],$ and write $T$ for $T_{0}$: thus $T$
is also the first strict ascending ladder time in $S.$ Results about the
tail behaviour of $T_{x}$ are known in three different regimes. Firstly,
with $U$ denoting the renewal function in the strict increasing ladder
process of $S,$ and with $x$ denoting any \textbf{fixed }continuity point of 
$U$\textbf{,} for any $\rho \in (0,1)$ the following statements are
equivalent:%
\begin{equation}
P(S_{n}>0)\rightarrow \rho \text{ as }n\rightarrow \infty :  \label{0.3}
\end{equation}%
\begin{equation}
P(T_{x}>n)\backsim U(x)n^{-\rho }L(n)\text{ as }n\rightarrow \infty .
\label{0.4}
\end{equation}%
(Here $L$ denotes a function which is slowly varying (s.v.) at $\infty ;$
its asymptotic behaviour is determined by the sequence $(\rho _{n,}n\geq 1),$
where $\rho _{n}=P(S_{n}>0),$ see e.g. \cite{rad3}.)\ \ The case $x=0$ of (%
\ref{0.4}) asserts that $T$ is in the domain of attraction of a positive
stable law of index $\rho :$ we write this as $T\in D(\rho ,1).$

In particular, (\ref{0.3}) and (\ref{0.4}) hold in the situation that $S$ is
in the domain of attraction of a strictly stable law without centreing (we
write $S\in D(\alpha ,\rho ),$ where $\alpha \in (0,2]$ is the index and $%
\rho \in (0,1)$ is the positivity parameter). In this asymptotically stable
case, if $c_{n}$ is such that $(S_{[nt]}/c_{n},t\geq 0)\overset{d}{%
\rightarrow }(Y_{t},t\geq 0),$ we can deduce from the functional central
limit theorem that, when $\ x_{n}:=x/c_{n}$ is bounded away from zero and
infinity,%
\begin{equation}
(T_{x}>n)\backsim P(\sigma _{x_{n}}>1)=\int_{1}^{\infty }h_{x_{n}}(t)dt,
\label{0.1}
\end{equation}%
where $h_{a}(\cdot )$ is the density function of $\sigma _{a}$, the first
passage time of the limiting stable process $Y$ over $a.$ Finally, if $%
\alpha \rho <1,$ so that $\overline{F},$ the right-hand tail of the
distribution function of $S_{1},$ is regularly varying with index $-\alpha ,$
with $\alpha \in (0,2),$ (we write this as $\overline{F}\in RV(-\alpha )$),
and $x/c_{n}\rightarrow \infty ,$ then it is known that%
\begin{equation}
P(T_{x}\leq n)=P(\max_{r\leq n}S_{r}>x)\backsim n\overline{F}(x).
\label{0.2}
\end{equation}

In this paper we will prove that in this asymptotically stable case \textbf{%
local} uniform versions of (\ref{0.4}), (\ref{0.1}), and (\ref{0.2}) hold in
the respective scenarios%
\begin{equation*}
\begin{array}{cc}
A: & x/c_{n}\rightarrow 0, \\ 
B: & x/c_{n}\text{ is bounded away from }0\text{ and }\infty , \\ 
C: & x/c_{n}\rightarrow \infty .%
\end{array}%
\end{equation*}

The inspiration for this programme comes from a recent paper by Vatutin and
Wachtel \cite{vw}, who show that in almost all cases that $S\in D(\alpha
,\rho )$ the following local estimate holds:%
\begin{equation}
P(T=n)\backsim \rho n^{-\rho -1}L(n)\text{ as }n\rightarrow \infty \text{ .}
\label{a}
\end{equation}%
(They actually show that (\ref{a}) can only fail if $S$ lives on a
non-centred lattice, when a modified version holds: we do not treat this
case.) The statement (\ref{a}) is a local version of the special case $x=0$
of (\ref{0.4}), and we mention at this point that their proof is quite
different according as $\alpha \rho <1$ or $\alpha \rho =1$. We also mention
that prior to \cite{vw}, the asymptotic behaviour of $P(T=n)$ was apparently
only known in the case of attraction to the Normal distribution: see \cite%
{epp} and \cite{ad2}. However the asymptotic behaviour of the ratio $%
P(T_{x}=n)/P(T=n)$ for fixed $x$ is known for strongly aperiodic recurrent
random walk on the integers, (see Theorem 7 of \cite{hk}), so our focus is
mainly on the case that $x\rightarrow \infty .$

Our first result shows that the obvious local version of (\ref{0.4}), viz%
\begin{equation}
P(T_{x}=n)\backsim \rho U(x)n^{-\rho -1}L(n)\text{ as }n\rightarrow \infty ,%
\text{ }  \label{0.5}
\end{equation}%
holds uniformly for $x\geq 0$ in case $A.$

In case $B,$ our result is a uniform local version of (\ref{0.1}), which is
valid in all cases.

Finally in case $C,$ we prove a uniform local version of (\ref{0.2}), but
this requires the additional assumption that $\alpha \rho <1,$ so that $%
\overline{F}\in RV(-\alpha ),$ and also a local version of this assumption$.$

To the best of our knowledge, these results are new for non-constant $x,$
except for the case of finite variance, where similar results were
established in Eppel \cite{epp}.

Our method of proof in cases $A$ and $B$ relies crucially on several
different local estimates of the distribution of $S_{n}$ conditional on $%
T_{x}>n,$ which extend results for the case $x=0$ from \cite{vw}, and in
case $C$ we use a conditional local limit theorem from \cite{ej}.

We state our notation, assumptions and results in detail in the next
section, then give some preliminary results in section 3, prove the
above-mentioned estimates, which may be of independent interest, in section
4, give a full proof of our main results in the lattice case in section 5,
and sketch the proof in the non-lattice case in the final section.

\section{Results}

\textbf{Notation }In what follows the phrase "$S$ is an a.s.r.w.",
(asymptotically stable random walk) will have the following meaning.

\begin{itemize}
\item $S=(S_{n},n\geq 0)$ is a 1-dimensional random walk with $S_{0}=0$ and $%
S_{n}=\sum_{1}^{n}X_{r}$ for $n\geq 1$ where $X_{1},X_{2},\cdots $ are
i.i.d. with $F(x)=P(X_{1}\leq x):$

\item $S$ is either non-lattice, or it takes values on the integers and is
aperiodic:

\item there is a monotone increasing continuous function $c(t)$ such that
the process defined by $X_{t}^{(n)}=S_{[nt]}/c_{n}$ converges weakly as $%
n\rightarrow \infty $ to a stable process $Y=(Y_{t},t\geq 0).$

\item the process $Y$ has index $\alpha \in (0,2],$ and $\rho
:=P(Y_{1}>0)\in (0,1).$
\end{itemize}

\begin{remark}
The case $\alpha \rho =1,\alpha \in (1,2]$ is the spectrally negative case,
and we will sometimes need to treat this case separately. If $\alpha \rho <1$
then $\alpha <2$ and the L\'{e}vy measure $\Pi $ of $Y$ has a density equal
to $c_{+}x^{-\alpha -1}$ on $(0,\infty )$ with $c_{+}>0,$ and then we can
also assume that the norming sequence satisfies%
\begin{equation}
n\overline{F}(c_{n})\rightarrow 1\text{ as }n\rightarrow \infty .  \label{z}
\end{equation}%
But if $\alpha \rho =1$ we will have 
\begin{equation}
n\overline{F}(c_{n})\rightarrow 0\text{ as }n\rightarrow \infty .  \label{w}
\end{equation}
\end{remark}

Here are our main results, where we recall that $h_{y}(\cdot )$ denotes the
density function of the passage time over level $y>0$ of the process $Y.$ We
will also adopt the convention that both $x$ and $\Delta $ are restricted to
the integers in the lattice case.

\begin{theorem}
\label{one} Assume that $S$ is an asrw. Then

(A) uniformly for $x$ such that $x/c_{n}\rightarrow 0,$ 
\begin{equation}
P(T_{x}=n)\text{ }\backsim U(x)P(T=n)\backsim \rho U(x)n^{-\rho -1}L(n)\text{
as }n\rightarrow \infty :  \label{t1}
\end{equation}

(B) uniformly in $x_{n}:=x/c(n)\in \lbrack D^{-1},D],$ for any $D>1,$%
\begin{equation}
P(T_{x}=n)\backsim n^{-1}h_{x_{n}}(1)\text{ as }n\rightarrow \infty .
\label{t2}
\end{equation}%
If, in addition, $\alpha \rho <1,$ and 
\begin{equation}
f_{x}^{\Delta }:=P(S_{1}\in \lbrack x,x+\Delta ))\text{ is regularly varying
as }x\rightarrow \infty ,  \label{add}
\end{equation}%
then

(C) uniformly for $x$ such that $x/c_{n}\rightarrow \infty ,$%
\begin{equation}
P(T_{x}=n)\backsim \overline{F}(x)\text{ as }n\rightarrow \infty .
\label{t3}
\end{equation}
\end{theorem}

From this we get immediately a strengthening of (\ref{0.4}).

\begin{corollary}
If $S$ is an a.s.r.w. the estimate 
\begin{equation*}
P(T_{x}>n)\backsim U(x)n^{-\rho }L(n)\text{ as }n\rightarrow \infty
\end{equation*}%
holds uniformly as $x/c_{n}\rightarrow 0.$
\end{corollary}

\begin{remark}
In view of \ (\ref{0.1}) the result (\ref{t2}) might seem obvious. However (%
\ref{0.1}) could also be written as $P(\max_{r\leq n}S_{r}\leq x)\backsim
\int_{0}^{x_{n}}m(y)dy,$ where $m$ denotes the density function of $%
\sup_{t\leq 1}Y_{s},$ and in the recent paper \cite{w}, Wachtel has shown
that the obvious local version of this is only valid under an additional
hypothesis.
\end{remark}

\begin{remark}
The asymptotic behaviour of $h_{x}(1)$ has been determined in \cite{ds}, and
is given by%
\begin{equation*}
h_{x}(1)\backsim k_{1}x^{\alpha \rho }\text{ as }x\downarrow 0,\text{ }%
h_{x}(1)\backsim k_{2}x^{-\alpha }\text{ as }x\rightarrow \infty .
\end{equation*}%
(We mention here that $k_{1,}k_{2},\cdots $ will denote particular fixed
positive constants whereas $C$ will denote a generic positive constant whose
value can change from line to line.) It is therefore possible to compare the
exact results in (\ref{t1}) and (\ref{t3}) with the behaviour of $%
n^{-1}h_{x_{n}}(1)$ when $x/c_{n}\rightarrow 0$ or $x/c_{n}\rightarrow
\infty .$ It turns out that the ratio of the two can tend to $0,$ or a
finite constant, or oscillate, \ depending on the s.v. functions involved
and the exact behaviour of $x,$ except in the special case that $%
c_{n}\backsim Cn^{\eta }.$ (Here, and throughout, we write $1/\alpha =\eta .$%
) In fact, if $\overline{F}(x)\backsim C/(x^{\alpha }L_{0}(x)),$ one can
check that, when $x/c_{n}\rightarrow \infty ,$ 
\begin{equation*}
\frac{n\overline{F}(x)}{h_{x_{n}}(1)}\backsim \frac{L_{0}(c_{n})}{L_{0}(x)},
\end{equation*}%
and of course $L_{0}$ is asymptotically constant only in the aforementioned
special case. Similarly, the RHS of (\ref{t1}) only has the same asymptotic
behaviour as $n^{-1}h_{x_{n}}(1)$ in this same special case.
\end{remark}

\begin{remark}
In the spectrally negative case $\alpha \rho =1,$ without further
assumptions we know little about the asymptotic behaviour of $\overline{F},$
so it is clear that (C) doesn't generally hold in this case, and indeed it
is somewhat surprising that parts (A) and (B) do hold.
\end{remark}

\section{Preliminaries}

Throughout this section it will be assumed that $S$ is an a.s.r.w.. With $%
(\tau _{0},H_{0}):=(0,0)$ we write $(\boldsymbol{\tau ,H})=((\tau
_{n},H_{n}),n\geq 0)$ for the bivariate renewal process of strict ladder
times and heights, so that $\tau _{1}=T$ and $H_{1}=S_{T}$ is the first
ladder height; we also write $\tau $ and $H$ for $\tau _{1}$ and $H_{1}$. It
is known that there are sequences $a_{n}$ and $b_{n}$ such that $(\tau
_{n}/a_{n},H_{n}/b_{n})$ converges in distribution to a bivariate law whose
marginals are positive stable laws with parameters $\rho $ and $\alpha \rho $
respectively, with the proviso that when $\alpha \rho =1$ we replace the
stable limit of $H_{n}/b_{n}$ by a point mass at 1. Thus $a,b,c$ are
regularly varying with indexes $\rho ^{-1},(\alpha \rho )^{-1},$ and $\eta $
respectively. Furthermore we can assume, without loss of generality, the
existence of continuous, increasing functions $a,b,c$ such that $%
a_{n}=a(n),b_{n}=b(n),c_{n}=c(n),$ and 
\begin{equation}
b(t)=k_{3}c(a(t)),\text{ }t\geq 0.  \label{x}
\end{equation}%
(See \cite{dg} for details).

Write $A(y)=\int_{0}^{\infty }P(H>y)dy$. We will find the following
consequence of (\ref{x}) useful.

\begin{lemma}
\label{five}There is a constant $k_{4}$ such that 
\begin{equation}
U(c_{n})\backsim \frac{c_{n}}{A(c_{n})}\backsim \frac{k_{4}}{P(\tau >n)}%
\text{ as }n\rightarrow \infty .  \label{k}
\end{equation}
\end{lemma}

\begin{proof}
The first statement is due to Erickson \cite{e}, and the second is a slight
reformulation of Lemma 13 of \cite{vw}, using the fact that $nP(\tau
>n)P(\tau ^{-}>n)\rightarrow k_{5},$ where $\tau ^{-}=\min \{n\geq
1:S_{n}\leq 0\}$ is the first weak decreasing ladder time.
\end{proof}

\begin{corollary}
\label{extra}If $V$ is the renewal function in the weak decreasing ladder
height process then there is a constant $k_{5}$ such that 
\begin{equation}
U(c_{n})V(c_{n})\backsim k_{6}n\text{ as }n\rightarrow \infty .  \label{l}
\end{equation}
\end{corollary}

\begin{proof}
This follows from $nP(\tau >n)P(\tau ^{-}>n)\rightarrow k_{5},$ (\ref{k}),
and the analogous statement about $V.$
\end{proof}

We will also need the following conditional functional limit theorem, in
which $X^{(n)}(t)=S_{\left\lfloor nt\right\rfloor }/c_{n},0\leq t\leq 1,$
where $\left\lfloor \cdot \right\rfloor $ denotes the integer part function,
and $S$ is an a.s.r.w.

\begin{proposition}
\label{two}Let $P_{x}$ denote the probability measure under which $S$ starts
at $x\geq 0,$ and put $T^{-}=\min (n\geq 1:S_{n}\leq 0),$

(i) If $x/c_{n}\rightarrow 0,$ then $P_{x}(X^{(n)}\in \cdot |T^{-}>t)$
converges weakly on the Skorohod space to $P(Z^{(1)}\in \cdot ),$ where $%
Z^{(1)}$ denotes the stable meander of length 1 based on $Y,$

(ii) If $x/c_{n}\rightarrow a>0,$ then $P_{x}(X^{(n)}\in \cdot |T^{-}>t)$
converges weakly on the Skorohod space to $P(a+Y\in \cdot |a+\inf_{s\leq
1}Y_{s}>0).$
\end{proposition}

\begin{proof}
(i) This is proved in \cite{rad4} for the special case $x\equiv 0,$ and it
is not difficult to deduce the quoted result by using the technique in
Section 5 of \cite{bd}. ( Actually the proof in \cite{bd} is for the case $%
\alpha =2,$ and concerns convergence to the Bessel process, rather than the
Brownian meander: but these processes are mutually absolutely continuous,
and so are their analogues for stable processes. We can therefore deduce
convergence to the meander from convergence to the process conditioned to
stay positive, as is done in a more general scenario in Section 4 of \cite%
{cd}.)

(ii) Since the probability of the limiting conditioning event is positive,
this follows from the weak convergence of $X^{(n)}$ to $Y.$
\end{proof}

\emph{Until further notice we assume we are in the lattice case, and }$%
x,y,z\cdots $ \emph{will be assumed to take non-negative integer values only.%
}

We will write%
\begin{equation*}
g(m,y)=\sum_{n=0}^{\infty }P(T_{n}=m,H_{n}=y)\text{ and }g(y)=\sum_{n=0}^{%
\infty }P(H_{n}=y)
\end{equation*}%
for the bivariate renewal mass function of $(\boldsymbol{\tau ,H})$ and the
renewal mass function of $\boldsymbol{H}$ respectively$.$ Our proofs are
based on the following obvious representation:%
\begin{equation}
P(T_{x}=n+1)=\sum_{y\geq 0}P(T_{x}>n,S_{n}=x-y)\overline{F}(y),\text{ }x\geq
0,n>0,  \label{i}
\end{equation}%
To exploit this we need good estimates of $P(T_{x}>n,S_{n}>x-y),$ and we
derive these from the formula 
\begin{equation}
P(S_{n}=x-y,T_{x}>n)=\sum_{z=0}^{y\wedge
x}\sum_{r=0}^{n}g(r,x-z)g^{-}(n-r,y-z),  \label{main}
\end{equation}%
where $g^{-}$ denotes the bivariate mass function in the weak downgoing
ladder process of $S.$ Formula (\ref{main}), which extends a result
originally due to Spitzer, follows by decomposing the event on the LHS
according to the time and position of the maximum, and using the well-known
duality result: 
\begin{equation}
g^{-}(m,u)=P(S_{m}=-u,\tau >m).  \label{d-}
\end{equation}%
(See Lemma 2.1 in \cite{ad}). Of course we also have%
\begin{equation}
g(m,u)=P(S_{m}=u,\tau ^{-}>m),  \label{d+}
\end{equation}%
and our main tool in estimating $P(S_{n}=x-y,T_{x}>n)$ will be the following
estimates for $g$ and $g^{-}.$ The results for $g$ are established in \cite%
{vw}, where they are stated as estimates for the conditional probability $%
P(S_{m}=u|\tau ^{-}>m).$ (See Theorems 5 and 6 in \cite{vw}.) The results
for $g^{-}$ can be derived by applying those results to $-S,$ and then using
the calculation given on page 100 of \cite{ad2} to deduce the result for the
weak ladder process. Recall that $V(x)=$ $\sum_{m=0}^{\infty
}\sum_{u=0}^{x}g^{-}(m,u)$, and write $f$ for the density of $Y_{1},$ where $%
Y$ is the limiting stable process. Also $p$ and $\tilde{p}$ stand for the
densities of $Z_{1}^{(1)}$ and $\tilde{Z}_{1}^{(1)},$ the stable meanders of
length $1$ at time $1$ corresponding to $Y$ and $-Y.$

\begin{lemma}
\label{nine}Uniformly in $x\geq 1$ and $x\geq 0,$ respectively,%
\begin{equation}
\frac{c_{n}g(n,x)}{P(\tau ^{-}>n)}=p(x/c_{n})+o(1)\text{ and }\frac{%
c_{n}g^{-}(n,x)}{P(\tau >n)}=\tilde{p}(x/c_{n})+o(1)\text{ as }n\rightarrow
\infty .  \label{l1}
\end{equation}%
Also, uniformly as $x/c_{n}\rightarrow 0,$%
\begin{equation}
g(n,x)\backsim \frac{f(0)U(x-1)}{nc_{n}}\text{ for }x\geq 1\text{ and }%
g^{-}(n,x)\backsim \frac{f(0)V(x)}{nc_{n}}\text{ for }x\geq 0.  \label{l2}
\end{equation}
\end{lemma}

From this we can deduce the following result, which is a minor extension of
Lemma 20 in \cite{vw}.

\begin{lemma}
\label{three}Given any constant $C_{1}$ there exists a constant $C_{2}$ such
that for all $n\geq 1$ and $0\leq x\leq C_{1}c_{n}$%
\begin{equation}
g(n,x)\leq \frac{C_{2}U(x)}{nc_{n}},\text{ and }g^{-}(n,x)\leq \frac{%
C_{2}V(x)}{nc_{n}}.  \label{e}
\end{equation}
\end{lemma}

.

\begin{proof}
Observe that on any interval $[\delta c_{n},C_{1}c_{n}],$ \ $p(x/c_{n})$ is
bounded, and $U(x)\geq U(\delta c_{n}),$ so by Lemma \ref{five} we see that
the ratio%
\begin{equation*}
\frac{P(\tau ^{-}>n)}{c_{n}}/\frac{U(x)}{nc_{n}}\leq nP(\tau ^{-}>n)U(\delta
c_{n})
\end{equation*}
is also bounded above. A similar proof works for $g^{-}.$
\end{proof}

Just as these local estimates for the distribution of $S_{n}$ on the event $%
\tau >n$ played a crucial r\^{o}le in the proof of (\ref{a}) in \cite{vw},
we need similar information on the event $T_{x}>n.$ This is given in the
following result, where for $x>0$ we write $q_{x}(\cdot )$ for the density
of $P(Y_{1}\in x-\cdot :\sup_{t\leq 1}Y_{t}<x).$ We also write $%
x/c_{n}=x_{n} $ and $y/c_{n}=y_{n}.$

\begin{proposition}
\label{four}(i) Uniformly as $x_{n}\vee y_{n}\rightarrow 0$%
\begin{equation}
P(S_{n}=x-y,T_{x}>n)\backsim \frac{U(x)f(0)V(y)}{nc_{n}}.  \label{A}
\end{equation}%
(ii) For any $D>1,$ uniformly for $y_{n}\in \lbrack D^{-1},D],$ 
\begin{equation}
P(S_{n}=x-y,T_{x}>n)\backsim \frac{U(x)P(\tau >n)\tilde{p}(y_{n})}{c_{n}}%
\text{ as }n\rightarrow \infty \text{ and }x_{n}\rightarrow 0,  \label{B}
\end{equation}%
and uniformly for $x_{n}\in \lbrack D^{-1},D],$ 
\begin{equation}
P(S_{n}=x-y,T_{x}>n)\backsim \frac{V(y)P(\tau ^{-}>n)p(x_{n})}{c_{n}}\text{
as }n\rightarrow \infty \text{ and }y_{n}\rightarrow 0.  \label{D}
\end{equation}%
(iii) For any $D>1,$ uniformly for $x_{n}\in \lbrack D^{-1},D]$ and $%
y_{n}\in \lbrack D^{-1},D],$%
\begin{equation}
P(S_{n}=x-y,T_{x}>n)\backsim \frac{q_{x_{n}}(y_{n})}{c_{n}}\text{ as }%
n\rightarrow \infty .  \label{C}
\end{equation}
\end{proposition}

The proof of this result is given in the next section. We can repeat the
argument used in Lemma \ref{three} to get the following corollary.

\begin{corollary}
\label{twelve}Given any constant $C_{1}$ there exists a constant $C_{2}$
such that for all $n\geq 1$ and $0\leq x\leq C_{1}c_{n},0\leq y\leq
C_{1}c_{n},$%
\begin{equation*}
P(S_{n}=x-y,T_{x}>n)\leq \frac{C_{2}U(x)f(0)V(y)}{nc_{n}}.
\end{equation*}
\end{corollary}

It is apparent that we will also need information about the behaviour of $%
g(n,x),$ or equivalently of $P(S_{n}=x,\tau ^{-}>n),$ in the case $%
x/c_{n}\rightarrow \infty $. Fortunately this has been obtained recently in 
\cite{ej}, and we quote Propositions 11 and 12 therein as (\ref{5.1}) and (%
\ref{5.3}). The related unconditional results (\ref{5.x}) and (\ref{5.y})
have been proved in special cases in \cite{rad1}, \cite{rad2}, and \cite{ej}%
, and the general results can be deduced from Theorem 2.1 of \cite{dds}.

\begin{proposition}
\label{thirteen}If $S$ is an asrw with $\alpha \rho <1,$ then, uniformly for 
$x$ such that $x/c_{n}\rightarrow \infty ,$%
\begin{equation}
P(S_{n}>x)\backsim nP(S_{1}>x)\text{ as }n\rightarrow \infty ,\text{ and}
\label{5.x}
\end{equation}%
\begin{equation}
\text{ }P(S_{n}>x,\tau ^{-}>n)\backsim \rho ^{-1}P(S_{n}>x)P(\tau ^{-}>n)%
\text{ as }n\rightarrow \infty .  \label{5.1}
\end{equation}%
If, additionally, (\ref{add}) holds, then%
\begin{equation}
P(S_{n}\in \lbrack x,x+\Delta ))\backsim nf_{x}^{\Delta }\text{ as }%
n\rightarrow \infty  \label{5.y}
\end{equation}%
and 
\begin{equation}
\text{ }P(S_{n}\in \lbrack x,x+\Delta ),\tau ^{-}>n)\backsim \rho
^{-1}nf_{x}^{\Delta }P(\tau ^{-}>n)\text{ as }n\rightarrow \infty .
\label{5.3}
\end{equation}
\end{proposition}

\subsection{Some identities for stable processes.}

\begin{proposition}
\label{id}(i) For any stable process $Y$ which has $\alpha \rho <1$ there
are positive constants $k_{7}$ and $k_{8}$ such that the following
identities hold: 
\begin{eqnarray}
h_{x}(t) &=&k_{7}\int_{0}^{\infty }q_{x}(w)w^{-\alpha }dw\text{ and}
\label{riv} \\
q_{u}(v) &=&k_{8}\dint\limits_{0}^{1}\dint\limits_{0}^{u\wedge
v}u(t,u-z)u^{-}(1-t,v-z)dzdt,  \label{jb}
\end{eqnarray}%
where $u$ and $u^{-}$ denote the bivariate renewal densities for the
increasing ladder processes of $Y$ and $-Y.$

(ii) If $\alpha \rho =1$ there is a positive constant $k_{9}$ such that 
\begin{equation}
p(x)=k_{9}h_{x}(1).  \label{lc}
\end{equation}
\end{proposition}

\begin{proof}
All three results are special cases of results for L\'{e}vy processes. The
general version of (\ref{riv}) is given in \cite{dr}, and (\ref{jb}) follows
from the following observation, which is a minor extension of Theorem 20,
p176 of \cite{jb}.\textbf{\ }

Assume $X$ is a L\'{e}vy process which is not compound Poisson. Then there
is a constant $k_{7}>0$ such that for $x>0$ and $w<x,$ 
\begin{equation}
P(X_{t}\in
dw,t<T_{x})dt=k_{7}\int_{y=w^{+}}^{x}\int_{s=0}^{t}U(ds,dy)U^{-}(dt-s,y-dw),
\label{D8}
\end{equation}%
where $U$ and $U^{-}$ are the bivariate renewal measures in the increasing
ladder processes of $X$ and $X^{-}.$ Clearly it suffices to prove that the
Laplace transform, in $t,$ of the LHS of (\ref{D8}) is the same as that of
the RHS, which is 
\begin{eqnarray}
&&k_{7}\text{ }\int_{y=w^{+}}^{x}\int_{t=0}^{\infty
}e^{-qt}\int_{s=0}^{t}U(ds,dy)U^{-}(dt-s,y-dw)  \notag \\
&=&k_{7}\text{ }\int_{y=w^{+}}^{x}\int_{t=0}^{\infty
}e^{-qt}U(dt,dy)\int_{s=0}^{\infty }e^{-qs}U^{-}(ds,y-dw).  \label{D7}
\end{eqnarray}%
Note that if we introduce an independent Exp($q$) random variable $%
\boldsymbol{e}_{q}$ the Wiener-Hopf factorisation allows us to write $X_{%
\boldsymbol{e}_{q}}=S_{\boldsymbol{e}_{q}}-\widetilde{S}_{\boldsymbol{e}%
_{q}}^{-},$ where $S$ denotes the supremum process of $X$ and $\tilde{S}^{-}$
denotes an independent copy of the supremum process $S^{-}$ of $-X.$ Let $%
\kappa $ and $\kappa ^{-}$ denote the bivariate Laplace exponents of the
ladder processes of $X$ and $X^{-}.$Then, using the identity $\kappa
(q,0)\kappa ^{-}(q,0)=q/k_{7}$ which follows from the Wiener-Hopf
factorisation, (see e.g. (3), p166 of \cite{jb}), 
\begin{eqnarray*}
\int_{t=0}^{\infty }e^{-qt}P(X_{t} &\in &dw,t<T_{x})dt \\
&=&q^{-1}P(X_{\boldsymbol{e}_{q}}\in dw,\boldsymbol{e}_{q}<T_{x}) \\
&=&q^{-1}P(S_{\boldsymbol{e}_{q}}\leq x,S_{\boldsymbol{e}_{q}}-\widetilde{S}%
_{\boldsymbol{e}_{q}}^{-}\in dw) \\
&=&q^{-1}\int_{w^{+}}^{x}P(S_{\boldsymbol{e}_{q}}\in dy)P(S_{\boldsymbol{e}%
_{q}}^{-}\in y-dw) \\
&=&k_{7}\int_{w^{+}}^{x}\frac{P(S_{\boldsymbol{e}_{q}}\in dy)}{\kappa (q,0)}%
\frac{P(S_{\boldsymbol{e}_{q}}^{-}\in y-dw)}{\kappa ^{-}(q,0)}.
\end{eqnarray*}%
But (1), p 163 of \cite{jb} gives%
\begin{equation*}
\frac{E\mathbb{(}e^{-\lambda S_{\boldsymbol{e}_{q}}})}{\kappa (q,0)}=\frac{1%
}{\kappa (q,\lambda )}=\int_{0}^{\infty }\int_{0}^{\infty }e^{-(qt+\lambda
y)}U(dt,dy),
\end{equation*}%
so 
\begin{equation*}
\frac{P(S_{\boldsymbol{e}_{q}}\in dy)}{\kappa (q,0)}=\int_{0}^{\infty
}e^{-qt}U(dt,dy).
\end{equation*}%
Using the analogous expression for $P(S_{\boldsymbol{e}_{q}}^{-}\in y-dw),$ (%
\ref{D7}) is immediate, and then (\ref{D8}) follows. Specializing this to
the stable case then gives (\ref{jb}). Finally, if we write $n$ for the
characteristic measure of the excursions away from zero of $X-I,$ with $I$
denoting the infimum process of $X$, then $p_{t}(dx):=n(\varepsilon _{t}\in
dx)/n(\zeta >t)$ is a probabilty measure which coincides with that of the
meander of length $t$ at time $t$ in the stable case. (Here $\zeta $ denotes
the life length of the generic excursion $\varepsilon .)$ In the special
case of spectrally negative L\'{e}vy processes, we have%
\begin{eqnarray*}
p_{t}(dx) &=&Cxt^{-1}P(X_{t}\in dx)/n(\zeta >t) \\
&=&Ch_{x}(t)dx/n(\zeta >t).
\end{eqnarray*}%
The first equality here comes from Cor 4 in \cite{ac}, and the second is the
L\'{e}vy version of the ballot theorem (see Corollary 3, p 190 of \cite{jb}%
). Specialising to the stable case and $t=1$ gives (\ref{lc}). (I owe this
observation to Loic Chaumont.)
\end{proof}

\section{\protect\bigskip Proof of Proposition \protect\ref{four}}

\begin{proof}
We will be applying the results in Lemma \ref{nine} to formula (\ref{main})
and we write the RHS of (\ref{main}) as $P_{1}+P_{2}+P_{3},$ where with $%
\delta \in (0,1/2),$ 
\begin{equation*}
P_{i}=\sum_{r\in A_{i}}\sum_{z=0}^{y\wedge x}g(r,x-z)g^{-}(n-r,y-z),\text{ }%
1\leq i\leq 3,
\end{equation*}%
and%
\begin{equation*}
A_{1}=\{0\leq r\leq \delta n\},\text{ }A_{2}=\{\delta n<r\leq (1-\delta )n\},%
\text{ and }A_{3}=\{(1-\delta )n<r\leq n\}.
\end{equation*}

(i) We introduce $d(\cdot ),$ an increasing and continuous function which
satisfies $d(n)=nc_{n},$ and with $\left\lfloor \cdot \right\rfloor $
standing for the integer part function, use Lemma \ref{nine} to write%
\begin{eqnarray*}
P_{1} &\sim &f(0)\sum_{z=0}^{y\wedge x}\sum_{r=0}^{\left\lfloor n\delta
\right\rfloor }g(r,x-z)\frac{V(y-z)}{d(n-r)} \\
&\leq &\frac{f(0)}{d(n(1-\delta ))}\sum_{z=0}^{y\wedge
x}V(y-z)\sum_{r=0}^{\left\lfloor n\delta \right\rfloor }g(r,x-z) \\
&=&\frac{f(0)}{d(n(1-\delta ))}\sum_{z=0}^{y\wedge
x}V(y-z)[u(x-z)-\sum_{r=\left\lfloor n\delta \right\rfloor +1}^{\infty
}g(r,x-z)].
\end{eqnarray*}%
We can apply Lemma \ref{nine} again to get the estimate, uniform for $%
w/c_{n}\rightarrow 0,$ 
\begin{equation*}
\sum_{r=\left\lfloor n\delta \right\rfloor +1}^{\infty }g(r,w)\sim
\sum_{r=\left\lfloor n\delta \right\rfloor +1}^{\infty }\frac{f(0)}{rc_{r}}%
U(w)\sim \frac{Cf(0)}{c_{n}}U(w).
\end{equation*}%
Since we know that 
\begin{equation*}
\lim \inf_{n\rightarrow \infty }\frac{nu(n)}{U(n)}>0,
\end{equation*}%
(see e.g. Theorem 8.7.4 in \cite{bgt}) we see that this is $o(u(w))$. Also
we have%
\begin{equation*}
\sum_{z=0}^{y\wedge x}\sum_{r=0}^{\left\lfloor n\delta \right\rfloor
}g(r,x-z)\frac{V(y-z)}{d(n-r)}\geq \frac{1}{d(n)}\sum_{z=0}^{y\wedge
x}\sum_{r=0}^{\left\lfloor n\delta \right\rfloor }g(r,x-z)V(y-z),
\end{equation*}%
so we see that 
\begin{equation}
\lim_{n,\delta }\frac{d(n)P_{1}}{f(0)\sum_{z=0}^{y\wedge x}V(y-z)u(x-z)}=1,
\label{p1}
\end{equation}%
where $\lim_{n,\delta }(\cdot )=1$ is shorthand for $\lim_{\delta
\rightarrow 0}\lim \sup_{n\rightarrow \infty }(\cdot )$ $=\lim_{\delta
\rightarrow 0}\lim \inf_{n\rightarrow \infty }(\cdot )=1.$ Similarly, once
we observe that 
\begin{eqnarray*}
P_{3} &=&\sum_{z=0}^{y\wedge x}\sum_{r=n-\left\lfloor n\delta \right\rfloor
}^{n}g(r,x-z)g^{-}(n-r,y-z) \\
&=&\sum_{z=0}^{y\wedge x}\sum_{r=0}^{\left\lfloor n\delta \right\rfloor
}g(n-r,x-z)g^{-}(r,y-z),
\end{eqnarray*}%
an entirely analogous argument gives%
\begin{equation}
\lim_{n,\delta }\frac{d(n)P_{3}}{f(0)\sum_{z=0}^{y\wedge x}U(x-z-1)v(y-z)}=1,
\label{p2}
\end{equation}%
where $v$ is the renewal mass function in the down-going ladder height
process. Noting that%
\begin{equation*}
U(x-z-1)v(y-z)+V(y-z)u(x-z)=U(x-z)V(y-z)-U(x-z-1)V(y-z-1),
\end{equation*}%
we get the formula%
\begin{equation*}
\sum_{z=0}^{y\wedge x}\{U(x-z-1)v(y-z)+V(y-z)u(x-z)\}=V(y)U(x),
\end{equation*}%
and the result will follow by letting $n\rightarrow \infty $ and then $%
\delta \downarrow 0$ provided $d(n)P_{2}=o(V(y)U(x))$ for each fixed $\delta
>0.$ In fact, using Lemma \ref{nine} again,%
\begin{eqnarray*}
d(n)P_{2} &=&d(n)\sum_{A_{2}}\sum_{z=0}^{y\wedge x}g(r,x-z)g^{-}(n-r,y-z) \\
&\sim &f(0)^{2}d(n)\sum_{z=0}^{y\wedge x}\sum_{A_{2}}\frac{V(y-z)U(x-z)}{%
d(r)d(n-r)} \\
&\leq &\frac{(y\wedge x)f(0)^{2}d(n)nV(y)U(x)}{d(\left\lfloor n\delta
\right\rfloor )^{2}}=O(\frac{(y\wedge x)V(y)U(x)}{c_{n}})=o(V(y)U(x)),
\end{eqnarray*}%
and the result follows.

(ii) In this case we can assume WLOG that $y\wedge x=x=o(y),$ so that Lemma %
\ref{nine} gives $g^{-}(n-r,y-z)\backsim P(\tau >n-r)\tilde{p}%
(y/c_{n-r})/c_{n-r}$ uniformly for $r\in A_{1}$and $0\leq z\leq x.$ With $e$
denoting a continuous and monotone interpolant of $c_{m}/P(\tau >m)$, and
noting that $\tilde{p}(\cdot )$ is uniformly continuous and bounded away
from zero and infinity on $[D^{-1},D],$ see \cite{ds}, we can use a similar
argument to that in (i) to show that 
\begin{eqnarray*}
P_{1} &=&\sum_{0}^{\left\lfloor \delta n\right\rfloor }\sum_{z=0}^{x}\frac{%
\tilde{p}(y/c_{n-r})}{e(n-r)}g(r,x-z)(1+o(1)) \\
&\leq &\frac{\tilde{p}^{\ast }(y)}{e(n(1-\delta ))}\left(
\sum_{z=0}^{x}u(x-z)(1+o(1))\right) , \\
&=&\frac{U(x)\tilde{p}^{\ast }(y)}{e(n(1-\delta ))}(1+o(1)),
\end{eqnarray*}%
where $\tilde{p}^{\ast }(y)=\sup_{0\leq r\leq \delta n}\tilde{p}(y/c_{n-r})=%
\tilde{p}(y_{n})+\varepsilon (n,\delta )$ and $\lim_{n,\delta }\varepsilon
(n,\delta )=0.$ \ In the same way we get $P_{1}\geq U(x)\tilde{p}_{\ast
}(y)/e(n)(1+o(1))$, where $\tilde{p}_{\ast }(y)=\inf_{0\leq r\leq \delta n}%
\tilde{p}(y/c_{n-r}),$ and we deduce that 
\begin{equation}
\lim_{n,\delta }\frac{e(n)P_{1}}{\tilde{p}(y_{n})U(x)}=1.  \label{p4}
\end{equation}%
Since $inf\{\tilde{p}(y):y\in \lbrack D^{-1},D]>0,$ the result will follow
if we can show that for any fixed $\delta >0$ 
\begin{equation}
\lim \sup_{n\rightarrow \infty }\frac{e(n)(P_{2}+P_{3})}{U(x)}=0.  \label{p3}
\end{equation}%
However (\ref{p2}) still holds, but note now that 
\begin{eqnarray*}
\sum_{z=0}^{y\wedge x}U(x-1-z)v(y-z) &=&\sum_{z=0}^{x}U(x-1-z)v(y-z) \\
&\leq &U(x)(V(y)-V(y-x-1))=o(U(x)V\left( y\right) ).
\end{eqnarray*}%
and since the analogue of (\ref{k}) holds, viz $V(c_{n})\backsim
k_{10}/P(\tau ^{-}>n),$ we see that%
\begin{eqnarray*}
\frac{e(n)P_{3}}{U(x)} &=&o\left( \frac{U(x)V\left( y\right) e(n)}{d(n)U(x)}%
\right) \\
&=&o(\frac{1}{nP(\tau >n)P(\tau ^{-}>n)})=o(1).
\end{eqnarray*}%
Finally%
\begin{eqnarray*}
e(n)P_{2} &=&e(n)\sum_{A_{2}}\sum_{z=0}^{x}g(r,x-z)g^{-}(n-r,y-z) \\
&\sim &f(0)e(n)\sum_{A_{2}}\sum_{z=0}^{x}\frac{U(x-z)\tilde{p}((y-z)/c_{n-r})%
}{d(r)e(n-r)} \\
&\leq &\frac{Ce(n)nxU(x)}{d(\left\lfloor n\delta \right\rfloor
)e(\left\lfloor n\delta \right\rfloor )}=O(\frac{xU(x)}{c_{n}})=o(U(x)).
\end{eqnarray*}%
Thus (\ref{p3}) is established, and the result (\ref{B}) follows. Since (\ref%
{D}) is (\ref{B}) for $-S$ with $x$ and $y$ interchanged, modified to take
account of the difference between strict and weak ladder epochs, we omit
it's proof.

(iii) In this case it is $P_{2}$ that dominates. To see this, note that if
we denote by $b(\cdot )$ a continuous interpolant of $c_{n}/P(\tau ^{-}>n),$
we have $b(n)e(n)\backsim nc_{n}^{2}.$ Then using Lemma \ref{nine} twice
gives, uniformly for $x_{n},y_{n}\in (D^{-1},D)$ and for any fixed $\delta
\in (0,1/2),$%
\begin{eqnarray*}
c_{n}P_{2} &=&c_{n}\sum_{A_{2}}\sum_{z=0}^{x\wedge y}\frac{p((x-z)/c_{r})%
\tilde{p}((y-z)/c_{n-r})}{b(r)e(n-r)}+o(c_{n}(x\wedge y)\sum_{A_{2}}\frac{1}{%
b(r)e(n-r)}) \\
&=&c_{n}\sum_{A_{2}}\sum_{z=0}^{x\wedge y}\frac{p((x-z)/c_{r})\tilde{p}%
((y-z)/c_{n-r})}{b(r)e(n-r)}+o(1).
\end{eqnarray*}%
Making the change of variables $r=nt,y=c_{n}z,$ recalling that $p$ and $%
\tilde{p}$ are uniformly continuous on compacts, and that%
\begin{equation*}
\frac{b(r)e(n-r)}{b(n)e(n)}\rightarrow t^{-\eta -1+\rho }(1-t)^{-1-\rho }
\end{equation*}%
uniformly on $A_{2},$ we see that for each fixed $\delta >0$ we get the
uniform estimate $c_{n}P_{2}=I_{\delta }(x_{n},y_{n})+o(1),$ where 
\begin{equation*}
I_{\delta }(u,v)=\dint\limits_{\delta }^{1-\delta }\dint\limits_{0}^{u\wedge
v}p(\frac{u-z}{t^{\eta }})\tilde{p}(\frac{v-z}{(1-t)^{\eta }})t^{-\eta
-1+\rho }(1-t)^{-\eta -\rho }dtdz.
\end{equation*}%
If we introduce $p_{t}(z)=t^{-\eta }p(zt^{-\eta }),$ which is the density
function of $Z_{t}$, the meander of length $t$ at time $t,$ according to
Lemma 8 of \cite{ds} we have that the renewal measure of the increasing
ladder process of $Y$ has a joint density which is given by $u(t,z)=Ct^{\rho
-1}p_{t}(z)=Ct^{-\eta +\rho -1}p(zt^{-\eta }).$ Similarly for the decreasing
ladder process \ we have $u^{-}(t,z)=C$ $t^{-\eta +\rho }\tilde{p}(zt^{-\eta
}),$ and (\ref{jb}) in Proposition \ref{id} gives 
\begin{equation*}
I_{0}(u,v)=C\dint\limits_{0}^{1}\dint\limits_{0}^{u\wedge
v}u(t,u-z)u^{-}(1-t,v-z)dzdt=Cq_{u}(v),
\end{equation*}%
so we conclude that%
\begin{equation}
\lim_{\delta \rightarrow 0}\lim_{n\rightarrow \infty }\frac{c_{n}P_{2}}{%
q_{x_{n}}(y_{n})}=C.  \label{p5}
\end{equation}%
Turning to $P_{1},$ if $K=\sup_{y\geq 0}\tilde{p}(y),$ we have%
\begin{eqnarray*}
c_{n}P_{1} &\backsim &c_{n}\sum_{0}^{\left\lfloor \delta n\right\rfloor
}\sum_{z=0}^{x\wedge y}\frac{\tilde{p}((y-z)/c_{n-r})}{e(n-r)}g(x-z,r) \\
&\leq &\frac{Kc_{n}}{e(n(1-\delta ))}\sum_{0}^{\left\lfloor \delta
n\right\rfloor }\sum_{z=0}^{x}g(x-z,r)\leq CP(\tau >n)\Gamma (\left\lfloor
\delta n\right\rfloor ),
\end{eqnarray*}%
where $\Gamma $ is the renewal function in the increasing \ ladder time
process. Since $T\in D(\rho ,1),$ we know that $\Gamma (\left\lfloor \delta
n\right\rfloor )\backsim \delta ^{\rho }\Gamma (n)$ and $P(\tau >n)\Gamma
(n)\rightarrow C,$ (see e.g. p 361 of \cite{bgt}), so we conclude that%
\begin{equation*}
\lim_{\delta \rightarrow 0}\lim \sup_{n\rightarrow \infty }c_{n}P_{1}=0.
\end{equation*}%
Exactly the same argument applies to $P_{3}$, and since $q_{u}(v)$ is
clearly bounded below by a positive constant for \thinspace $u,v\in \lbrack
D^{-1},D]$ we have shown that (\ref{C}) holds, except that the RHS is
multiplied by some constant $C.$ However if $C\neq 1$, by summing over $y$
we easily get a contradiction, and this finishes the proof.
\end{proof}

\section{Proof of Theorem \protect\ref{one}}

\subsection{Proof when $x/c_{n}\rightarrow 0.$}

\begin{proof}
As already indicated, the proof involves applying the estimates in
Proposition \ref{four} to the representation (\ref{i}), which we recall is%
\begin{equation}
P(T_{x}=n+1)=\sum_{y\geq 0}P(S_{n}=x-y,T_{x}>n)\overline{F}(y),\text{ }x\geq
0,n>0.  \label{j}
\end{equation}%
In the case $\alpha \rho <1$, given $\varepsilon >0$ we can find $%
K_{\varepsilon }$ and $n_{\varepsilon }$ such that $n\overline{F}%
(K_{\varepsilon }c_{n})\leq \varepsilon $ for \ $n\geq n_{\varepsilon }$
and, using (\ref{A}) from Proposition \ref{four}, 
\begin{equation}
P(S_{n}=x-y,T_{x}>n)\leq \frac{2U(x)f(0)V(y)}{nc_{n}}\text{ for \ all }x\vee
y\leq \varepsilon c_{n}\text{ and }n\geq n_{\varepsilon }.  \label{q1}
\end{equation}%
We can then use (\ref{B}) of Proposition \ref{four} to show that we can also
assume, increasing the value of $n_{\varepsilon }$ if necessary, that for $%
x\leq \varepsilon c_{n}$ and $y\in (\varepsilon c_{n},K_{\varepsilon
}c_{n}), $%
\begin{equation}
1-\varepsilon \leq \frac{c_{n}P(S_{n}=x-y,T_{x}>n)}{U(x)P(\tau >n)\tilde{p}%
(y_{n})}\leq 1+\varepsilon \text{ for all }n\geq n_{\varepsilon }.
\label{q2}
\end{equation}%
For fixed $\varepsilon $ it is clear that as $n\rightarrow \infty $ 
\begin{equation}
\sum_{\varepsilon c_{n}<y<K_{\varepsilon }c_{n}}\tilde{p}(y_{n})\overline{F}%
(y)/c_{n}\backsim \int_{\varepsilon }^{K_{\varepsilon }}\tilde{p}(z)%
\overline{F}(c_{n}z)dz\backsim n^{-1}\int_{\varepsilon }^{K_{\varepsilon }}%
\tilde{p}(z)z^{-\alpha }dz.  \label{q3}
\end{equation}%
Since it is known (see (109) in \cite{vw} or Proposition 10 in \cite{ds})
that $k_{11}:=\int_{0}^{\infty }z^{-\alpha }\tilde{p}(z)dz<\infty ,$ and we
can assume $K_{\varepsilon }\uparrow \infty $ as $\varepsilon \downarrow 0$
we see that, provided 
\begin{equation}
\lim_{\varepsilon \downarrow 0}\lim \sup_{n\rightarrow \infty }\frac{1}{%
U(x)P(\tau =n)}\sum_{y\in \lbrack 0,\varepsilon c_{n}]\cup \lbrack
K_{\varepsilon }c_{n},\infty )}P(S_{n}>x-y,T_{x}>n)\overline{F}(y)=0,
\label{q4}
\end{equation}%
it will follow from (\ref{q2}) that $P(T_{x}=n)\backsim U(x)k_{11}\rho
^{-1}P(\tau =n).$ Since this holds in particular for $x=0,$ we see that $%
k_{11}=\rho ,$ so it remains only to verify (\ref{q4}). Note first that%
\begin{eqnarray*}
\sum_{y\in \lbrack K_{\varepsilon }c_{n},\infty )}P(S_{n} &=&x-y,T_{x}>n)%
\overline{F}(y)\leq \overline{F}(K_{\varepsilon }c_{n})P(T_{x}>n) \\
&\leq &\varepsilon n^{-1}P(T_{x}>n),
\end{eqnarray*}%
so one part of (\ref{q4}) will follow if we can show that 
\begin{equation}
P(T_{x}>n)\leq CU(x)P(\tau >n).  \label{xx}
\end{equation}%
Obviously, for any $B>0$%
\begin{eqnarray*}
P(T_{x} &>&n)=\sum_{y\geq 0}P(S_{n}=x-y,T_{x}>n) \\
&\leq &\sum_{0\leq y\leq
Bc_{n}}P(S_{n}=x-y,T_{x}>n)+P(S_{n}<x-Bc_{n},T_{x}>n).
\end{eqnarray*}%
Since $x/c_{n}\rightarrow 0,$ we can use the invariance principle in
Proposition \ref{two} to fix $B$ large enough to ensure that $%
P(S_{n}<x-Bc_{n}|T_{x}>n)\leq 1/2$ for all sufficiently large $n.$ We can
also use Corollary \ref{twelve} to get%
\begin{eqnarray*}
\sum_{0\leq y\leq Bc_{n}}P(S_{n} &=&x-y,T_{x}>n)\leq \frac{CU(x)}{nc_{n}}%
\sum_{0\leq y\leq Bc_{n}}V(y) \\
&\leq &\frac{CU(x)Bc_{n}V(Bc_{n})}{nc_{n}}\leq \frac{CU(x)}{nP(\tau ^{-}>n)}
\\
&\leq &CU(x)P(\tau >n),
\end{eqnarray*}%
and thus (\ref{xx}) is established. (Note that this proof is also valid for
the case $\alpha \rho =1.$) Now (\ref{q1}) gives%
\begin{eqnarray*}
\lim_{\varepsilon \downarrow 0}\lim \sup_{n\rightarrow \infty } &&\frac{1}{%
U(x)P(\tau =n)}\sum_{y\in \lbrack 0,\varepsilon c_{n}]}P(S_{n}>x-y,T_{x}>n)%
\overline{F}(y) \\
&\leq &\lim_{\varepsilon \downarrow 0}\lim \sup_{n\rightarrow \infty }\frac{%
2f(0)}{nc_{n}P(\tau =n)}\sum_{y\in \lbrack 0,\varepsilon c_{n}]}\overline{F}%
(y)V(y),
\end{eqnarray*}%
and since $\alpha \rho <1,$ (105) of \cite{vw} shows that this is zero.

In the case $\alpha \rho =1,$ a different proof is required. We make use of
the observation in \cite{vw} that there is a sequence $\delta _{n}\downarrow
0$ with $\delta _{n}c_{n}\rightarrow \infty $ and such that $n\overline{F}%
(\delta _{n}c_{n})\rightarrow 0.$ This gives 
\begin{eqnarray*}
\sum_{y>\delta _{n}c_{n}}P(S_{n} &=&x-y,T_{x}>n)\overline{F}(y)\leq
P(T_{x}>n)\overline{F}(\delta _{n}c_{n}) \\
&=&o(U(x)P(\tau >n)/n)=o(U(x)P(\tau =n)),
\end{eqnarray*}%
where we have used (\ref{xx}). Using (\ref{A}) of Proposition \ref{four}
gives \ 
\begin{equation*}
\sum_{y\leq \delta _{n}c_{n}}P(S_{n}=x-y,T_{x}>n)\overline{F}(y)\backsim 
\frac{f(0)U(x)\omega (n)}{nc_{n}}\text{ as }n\rightarrow \infty ,
\end{equation*}%
uniformly in $x$, where $\omega (n)=\sum_{0\leq y\leq \delta _{n}c_{n}}V(y)%
\overline{F}(y).$ But in \cite{vw}, it is shown that $nc_{n}P(\tau
=n)\backsim f(0)\omega (n)$, so we deduce from (\ref{j}) that $%
P(T_{x}=n+1)\backsim U(x)P(\tau =n),$ as required.
\end{proof}

\bigskip

\subsection{Proof when $x/c_{n}=O(1)$}

\begin{proof}
Again we treat the case $\alpha \rho <1$ first$,$ and start by noting that
for any $B>0,$%
\begin{equation*}
\sum_{y>Bc_{n}}P(S_{n}=x-y,T_{x}>n)\overline{F}(y)\leq \overline{F}%
(Bc_{n})\backsim \frac{1}{nB^{\alpha }}\text{ as }n\rightarrow \infty ,
\end{equation*}%
uniformly in $x\geq 0.$ Also by Corollary \ref{twelve}, for $b>0$, 
\begin{equation}
\sum_{y\leq bc_{n}}P(S_{n}=x-y,T_{x}>n)\overline{F}(y)\leq \frac{%
CU(x)\sum_{y\leq bc_{n}}V(y)\overline{F}(y)}{nc_{n}}.  \label{q6}
\end{equation}%
Since $V\overline{F}\in RV(-\alpha \rho ),$ $\sum_{y\leq z}V(y)\overline{F}%
(y)\backsim zV(z)\overline{F}(z)/(1-\alpha \rho ),$ and we see that when $%
x\leq Dc_{n},$ 
\begin{eqnarray*}
\frac{U(x)\sum_{y\leq bc_{n}}V(y)\overline{F}(y)}{c_{n}} &\backsim
&b^{1-\alpha \rho }U(x)V(c_{n})\overline{F}(c_{n}) \\
&\leq &Cb^{1-\alpha \rho }D^{\alpha \rho }U(c_{n})V(c_{n})/n\leq
Cb^{1-\alpha \rho }D^{\alpha \rho },
\end{eqnarray*}%
where we have used (\ref{l}) from Corollary \ref{extra}. We conclude the
proof by showing that%
\begin{equation}
\lim_{b\downarrow 0,\text{ }B\uparrow \infty }\lim_{n\rightarrow \infty }%
\frac{n}{h_{x_{n}}(1)}\sum_{bc_{n}<y\leq Bc_{n}}P(S_{n}=x-y,T_{x}>n)%
\overline{F}(y)=C,  \label{q7}
\end{equation}%
uniformly for $x_{n}\in \lbrack D^{-1},D],$ since this would contradict (\ref%
{0.1}) if $C$ differed from 1. But in fact (\ref{q7}) follows immediately
from (\ref{C}) and the identity (\ref{riv}) in Proposition \ref{id}. When $%
\alpha \rho =1,$ it is immediate that%
\begin{equation*}
n\sum_{y>\delta _{n}c_{n}}P(S_{n}=x-y,T_{x}>n)\overline{F}(y)\leq n\overline{%
F}(\delta _{n}c_{n})\rightarrow 0,
\end{equation*}%
and%
\begin{eqnarray*}
&&n\sum_{0\leq y\leq \delta _{n}c_{n}}P(S_{n}=x-y,T_{x}>n)\overline{F}(y) \\
&\backsim &\frac{nP(\tau ^{-}>n)p(x_{n})}{c_{n}}\sum_{0\leq y\leq \delta
_{n}c_{n}}V(y)\overline{F}(y) \\
&\backsim &\frac{np(x_{n})\omega (n)}{P(\tau >n)nc_{n}}\backsim \frac{%
p(x_{n})P(\tau =n)}{P(\tau >n)}\backsim \rho p(x_{n}),
\end{eqnarray*}%
and the result follows from the identity (\ref{lc}) in Proposition \ref{id},
since again there would be a contradiction if $\rho k_{9}\neq 1$.
\end{proof}

\subsection{Proof when $x/c_{n}\rightarrow \infty $}

\begin{proof}
This time we write $P(T_{x}=n+1)=\sum_{1}^{4}P^{(i)},$ where $P^{(i)}=$%
\newline
$P\{A^{(i)}\cap (T_{x}=n+1)\}$ and 
\begin{eqnarray*}
A^{(1)} &=&(S_{n}\leq \delta x),\text{ }A^{(2)}=(\delta x<S_{n}\leq
x-Kc_{n}),\text{ } \\
A^{(3)} &=&(x-Kc_{n}<S_{n}\leq x-\gamma c_{n}),\text{ and }%
A^{(4)}=(S_{n}>x-\gamma c_{n}).
\end{eqnarray*}%
We note first that for $\delta \in (0,1),$ 
\begin{eqnarray}
\lim \sup_{n\rightarrow \infty }\frac{P^{(1)}}{\overline{F}(x)} &\leq &\lim
\sup_{n\rightarrow \infty }\frac{P(S_{n}\leq \delta x)\overline{F}((1-\delta
)x)}{\overline{F}(x)}=(1-\delta )^{-\alpha },  \notag \\
\lim \inf_{n\rightarrow \infty }\frac{P^{(1)}}{\overline{F}(x)} &\geq &\lim
\inf_{n\rightarrow \infty }\frac{P(\max_{r\leq n}S_{r}\leq x,|S_{n}|\leq
\delta x)\overline{F}((1+\delta )x)}{\overline{F}(x)}  \notag \\
&=&(1+\delta )^{-\alpha }.  \label{5.4}
\end{eqnarray}%
Next, using (\ref{5.x}), 
\begin{eqnarray}
\lim \sup_{n\rightarrow \infty }\frac{P^{(2)}}{\overline{F}(x)} &\leq &\lim
\sup_{n\rightarrow \infty }\frac{P(S_{n}>\delta x)\overline{F}(Kc_{n})}{%
\overline{F}(x)}  \notag \\
&\leq &\lim \sup_{n\rightarrow \infty }\frac{n\overline{F}(\delta x)%
\overline{F}(c_{n})}{\overline{F}(x)K^{\alpha }}=\frac{C}{(\delta K)^{\alpha
}}.  \label{5.5}
\end{eqnarray}%
To deal with the next term, we use (\ref{5.x}) again to see that for any
fixed $K,$ $P(x-Kc_{n}<S_{n}\leq x)$ is uniformly $o(n\overline{F}(x)).$
Since $P^{(3)}\leq \overline{F}(\gamma c_{n})P(x-Kc_{n}<S_{n}\leq x),$ we
deduce that%
\begin{equation}
\frac{P^{(3)}}{\overline{F}(x)}\rightarrow 0\text{ uniformly for each fixed }%
\gamma \text{ and }K.  \label{5.6}
\end{equation}%
As we are in the lattice case, (\ref{add}) tells us that $%
f(x):=P(S_{1}=x)\backsim \alpha x^{-1}\overline{F}(x),$ and combining this
with (\ref{5.3}) gives $g(n,x)=P(S_{n}=x,\tau ^{-}>n)\backsim $\newline
$\alpha nP(\tau ^{-}>n)x^{-1}\overline{F}(x),$ so we can assume that%
\begin{equation}
\sup_{n>0,\text{ }x>Kc_{n}}\frac{xg(n,x)P(\tau >n)}{\overline{F}(x)}<\infty .
\label{bound}
\end{equation}%
Then for large enough $n$ and $x>Kc_{n}$ 
\begin{eqnarray*}
P^{(4)} &=&\sum_{0}^{n}\sum_{y=0}^{\left\lfloor \gamma c_{n}\right\rfloor
}\sum_{z=0}^{\left\lfloor \gamma c_{n}\right\rfloor -y}g(n-m,x-y)g^{-}(m,z)%
\overline{F}(y+z) \\
&\leq &\frac{C\overline{F}(x)}{xP(\tau >n)}\sum_{0}^{n}\sum_{y=0}^{\left%
\lfloor \gamma c_{n}\right\rfloor }\sum_{z=0}^{\left\lfloor \gamma
c_{n}\right\rfloor -y}g^{-}(m,z)\overline{F}(y+z) \\
&\leq &\frac{C\overline{F}(x)}{xP(\tau >n)}\sum_{y=0}^{\left\lfloor \gamma
c_{n}\right\rfloor }\sum_{z=0}^{\left\lfloor \gamma c_{n}\right\rfloor
-y}v(z)\overline{F}(y+z) \\
&=&\frac{C\overline{F}(x)}{xP(\tau >n)}\sum_{z=0}^{\left\lfloor \gamma
c_{n}\right\rfloor }\sum_{y=0}^{\left\lfloor \gamma c_{n}\right\rfloor
-z}v(z)\overline{F}(y+z) \\
&=&\frac{C\overline{F}(x)}{xP(\tau >n)}\sum_{z=0}^{\left\lfloor \gamma
c_{n}\right\rfloor }\sum_{w=z}^{\left\lfloor \gamma c_{n}\right\rfloor }v(z)%
\overline{F}(w).
\end{eqnarray*}%
Now a summation by parts and the fact that $V\overline{F}\in RV(-\alpha \rho
)$ shows that as $y\rightarrow \infty $%
\begin{equation*}
\sum_{z=0}^{y}\sum_{w=z}^{y}v(z)\overline{F}(w)\backsim \sum_{z=0}^{y}V(z)%
\overline{F}(z)\backsim yV(y)\overline{F}(y)/(1-\alpha \rho ).
\end{equation*}%
So for all large enough $n$ we have the bound%
\begin{eqnarray}
P^{(4)} &\leq &\frac{C\overline{F}(x)\gamma ^{1-\alpha \rho }c_{n}V(c_{n})}{%
nxP(\tau >n)}  \notag \\
&\leq &\frac{C\overline{F}(x)\gamma ^{1-\alpha \rho }c_{n}}{nxP(\tau
>n)P(\tau ^{-}>n)}\backsim \frac{C\gamma ^{1-\alpha \rho }c_{n}}{x}\cdot 
\overline{F}(x)  \label{5.7}
\end{eqnarray}%
The result follows from (\ref{5.4})-(\ref{5.7}) and appropriate choice of $%
\delta ,K,$ and $\gamma .$
\end{proof}

\begin{remark}
The assumption (\ref{add}) is not strictly necessary for (\ref{bound}) to
hold, and this is the only point where we use this assumption. In fact, if
the following slightly weaker version of (\ref{bound}),%
\begin{equation}
\sup_{n>0,\text{ }x>Kc_{n}}\frac{c_{n}g(n,x)P(\tau >n)}{\overline{F}(x)}%
<\infty ,  \label{wbound}
\end{equation}%
were to hold, then (\ref{5.7}) would hold with $c_{n}$ replacing $x$ in the
denominator, and the proof would still be valid, by choosing $\gamma $ small.
\end{remark}

\section{The non-lattice case}

We indicate here the main differences between the proof in the lattice and
non-lattice cases. First, we have%
\begin{eqnarray*}
G(n,dy) &:&=\sum_{r=0}^{\infty }P(T_{r}=n,H_{r}\in dy)=P(S_{n}\in dy,\tau
^{-}>n), \\
G^{-}(n,dy) &:&=\sum_{r=0}^{\infty }P(T_{r}^{-}=n,H_{r}^{-}\in
dy)=P(-S_{n}\in dy,\tau >n),
\end{eqnarray*}%
and the following analogue of Lemma \ref{nine} is given in Theorems 3 and 4
of \cite{vw}.

\begin{lemma}
\label{fourteen}For any $\Delta _{0}>0,$ uniformly in $x\geq 0$ and $%
0<\Delta \leq \Delta _{0},$%
\begin{equation}
\frac{c_{n}G(n,[x,x+\Delta ])}{P(\tau ^{-}>n)}=\Delta p(x/c_{n})+o(1)\text{
and }\frac{c_{n}G^{-}(n,[x,x+\Delta ])}{P(\tau >n)}=\Delta \tilde{p}%
(x/c_{n})+o(1)\text{ as }n\rightarrow \infty .  \label{6.1}
\end{equation}%
Also, uniformly as $x/c_{n}\rightarrow 0,$%
\begin{equation}
G(n,[x,x+\Delta ])\backsim \frac{f(0)\int_{x}^{x+\Delta }U(w)dw}{nc_{n}}%
\text{ and }G^{-}(n,[x,x+\Delta ])\backsim \frac{f(0)\int_{x}^{x+\Delta
}V(w)dw}{nc_{n}}.  \label{6.2}
\end{equation}
\end{lemma}

\begin{remark}
Again, only the results for $G$ are given in \cite{vw}, but it is easy to
get the reults for $G^{-}.$ Actually the result in \cite{vw} has $U(w-)$
rather than $U(w)$ in (\ref{6.2}), but clearly the two integrals coincide.
Finally the uniformity in $\Delta $ is not mentioned in \cite{vw}, but a
perusal of the proof shows that this is true, essentially because it holds
in Stone's local limit theorem. See e.g. Theorem 8.4.2 in \cite{bgt}.
\end{remark}

In writing down the analogues of (\ref{i}) and (\ref{main}) care is required
with the the limits of integration, since the distribution of $S_{n}$ and
the renewal measures are not necessarily diffuse. These analogues are%
\begin{equation}
P(T_{x}=n+1)=\int_{[0,\infty )}P(S_{n}\in x-dy,T_{x}>n)\overline{F}(y),
\label{6.3}
\end{equation}%
and for $w\geq 0$%
\begin{equation}
P(S_{n}\in x-dw,T_{x}>n)=\sum_{r=0}^{n}\int_{[0,x)\cap \lbrack
0,w]}G(r,x-dz)G^{-}(n-r,dw-z).  \label{6.9}
\end{equation}%
The key result, the analogue of Proposition \ref{four}, is

\begin{proposition}
\label{fifteen}Fix $\Delta >0.$ Then (i) uniformly as $x_{n}\vee
y_{n}\rightarrow 0,$%
\begin{equation}
P(S_{n}\in (x-y-\Delta ,x-y],T_{x}>n)\backsim \frac{U(x)f(0)\int_{y}^{y+%
\Delta }V(w)dw}{nc_{n}}.  \label{1}
\end{equation}%
(ii) For any $D>1,$ uniformly for $y_{n}\in \lbrack D^{-1},D],$ 
\begin{equation}
P(S_{n}\in (x-y-\Delta ,x-y],T_{x}>n)\backsim \frac{U(x)P(\tau >n)\Delta 
\tilde{p}(y_{n})}{c_{n}}\text{ as }n\rightarrow \infty \text{ and }%
x_{n}\rightarrow 0,  \label{2}
\end{equation}%
and uniformly for $x_{n}\in \lbrack D^{-1},D],$ 
\begin{equation}
P(S_{n}\in (x-y-\Delta ,x-y],T_{x}>n)\backsim \frac{V(y)P(\tau ^{-}>n)\Delta
p(x_{n})}{c_{n}}\text{ as }n\rightarrow \infty \text{ and }y_{n}\rightarrow
0.  \label{3}
\end{equation}%
(iii) For any $D>1,$ uniformly for $x_{n}\in \lbrack D^{-1},D]$ and $%
y_{n}\in \lbrack D^{-1},D],$%
\begin{equation}
P(S_{n}\in (x-y-\Delta ,x-y],T_{x}>n)\backsim \frac{\Delta q_{x_{n}}(y_{n})}{%
c_{n}}\text{ as }n\rightarrow \infty .  \label{4}
\end{equation}
\end{proposition}

Once we have these results, we deduce Theorem \ref{one} by applying them to
a modified version of (\ref{6.3}), viz%
\begin{eqnarray*}
\sum_{0}^{\infty }P(S_{n} &\in &(x-(n+1)\Delta ,x-n\Delta ],T_{x}>n)%
\overline{F}(n\Delta )\geq P(T_{x}=n+1) \\
&\geq &\sum_{0}^{\infty }P(S_{n}\in (x-(n+1)\Delta ,x-n\Delta ],T_{x}>n)%
\overline{F}((n+1)\Delta ),
\end{eqnarray*}%
and letting $\Delta \rightarrow 0.$ So the key step is establishing
Proposition \ref{fifteen}, and we illustrate how this can be done by proving
(\ref{1}).

\begin{proof}
We want to apply Lemma \ref{fourteen} to (\ref{6.9}), but technically the
problem is that we can't do this directly, as we did in the lattice case.
The first step is to get an integrated form of (\ref{6.9}), and it is useful
to separate off the term $r=0$ $,$ so that\ for $x,y\geq 0,$%
\begin{equation}
P(S_{n}\in (x-y-\Delta ,x-y],T_{x}>n)=G^{-}(n,[\{y-x\}^{+},y+\Delta -x))%
\boldsymbol{1}_{\{x\leq y+\Delta \}}+\tilde{P},  \label{6.10}
\end{equation}%
where 
\begin{eqnarray}
\tilde{P} &=&\sum_{r=1}^{n}\int_{y\leq w\,<y+\Delta }\int_{z\in \lbrack
0,x)\cap \lbrack 0,w]}G(r,x-dz)G^{-}(n-r,dw-z),  \notag \\
&=&\sum_{r=1}^{n}\int_{0\leq z\,<x\wedge (y+\Delta )}\int_{y\vee z\leq
w<(y+\Delta )}G(r,x-dz)G^{-}(n-r,dw-z)  \notag \\
&=&\sum_{r=1}^{n}\int_{0\leq z\,<x\wedge (y+\Delta
)}G(r,x-dz)G^{-}(n-r,[(y-z)^{+},y+\Delta -z)),  \label{6.11}
\end{eqnarray}%
Using\ a similar notation as in the proof of Proposition \ref{four} we split 
$\tilde{P}$ into three terms, and note first from (\ref{6.2}) and (\ref{6.11}%
) that 
\begin{eqnarray*}
\tilde{P}_{1} &\backsim &f(0)\sum_{r=1}^{\left\lfloor n\delta \right\rfloor }%
\frac{1}{d(n-r)}\int_{0\leq z\,<x\wedge (y+\Delta
)}G(r,x-dz)\int_{(y-z)^{+}}^{y+\Delta -z}V(u)du \\
&\leq &\frac{f(0)}{d(n(1-\delta ))}\int_{0\leq z\,<x\wedge (y+\Delta
)}U(x-dz)\int_{(y-z)^{+}}^{y+\Delta -z}V(u)du.
\end{eqnarray*}%
An asymptotic lower bound is given by 
\begin{equation*}
\frac{f(0)}{d(n)}\sum_{r=1}^{\left\lfloor n\delta \right\rfloor }\int_{0\leq
z\,<x\wedge (y+\Delta )}G(r,x-dz)\int_{(y-z)^{+}}^{y+\Delta -z}V(u)du,
\end{equation*}
and it is easy to see that%
\begin{equation*}
\sum_{r>n\delta }\int_{0\leq z\,<x\wedge (y+\Delta
)}G(r,x-dz)\int_{(y-z)^{+}}^{y+\Delta -z}V(u)du=o(U(x)V_{\Delta }(y)),
\end{equation*}%
where we have put $V_{\Delta }(y):=\int_{y}^{y+\Delta }V(u)du.$ Noting that $%
U(x-dz)=\sum_{1}^{\infty }G(r,x-dz)$ for $0\leq z<x,$ this leads to a
similar uniform asymptotic lower bound, and hence that%
\begin{equation}
\lim_{n,\delta }\frac{d(n)P_{1}}{f(0)\int_{0\leq z\,<x\wedge (y+\Delta
)}U(x-dz)\int_{(y-z)^{+}}^{y+\Delta -z}V(u)du}=1.  \label{6.5}
\end{equation}

Dealing with $\tilde{P}_{3}$ is more complicated. First we write%
\begin{eqnarray*}
\tilde{P}_{3} &=&\sum_{r=0}^{\left\lfloor n\delta \right\rfloor }\int_{0\leq
z\,<x\wedge (y+\Delta )}G(n-r,x-dz)G^{-}(r,[(y-z)^{+},y+\Delta -z)) \\
&=&\sum_{r=0}^{\left\lfloor n\delta \right\rfloor }\int_{x\wedge (y+\Delta
)-x<w\leq x}G(n-r,dw)G^{-}(r,[(y-x+w)^{+},y+\Delta -x+w)).
\end{eqnarray*}%
We approximate this below and above by breaking the range of integration
into subintervals of length $\varepsilon \ll \Delta ,$ then use the estimate 
$G(n-r,[k\varepsilon ,(k+1)\varepsilon ))\backsim f(0)\int_{k\varepsilon
}^{(k+1)\varepsilon }U(v)dv/d(n-r),$ and finally let $\varepsilon
\rightarrow 0$ to conclude that%
\begin{equation}
\lim_{n,\delta }\frac{d(n)P_{3}}{f(0)\int_{0\leq z\,<x\wedge (y+\Delta
)}U(x-z)dz\int_{y\vee z\leq w<y+\Delta }V(du-z)}=1.  \label{6.6}
\end{equation}%
(Note that the term corresponding to $r=n$ in (\ref{6.11}) is included here.)

Also, for any fixed $\delta \in (0,1/2)$ we can use (\ref{6.2}) twice to see
that%
\begin{eqnarray}
\tilde{P}_{2} &\backsim &f(0)\sum_{\left\lfloor n\delta \right\rfloor
<r<\left\lfloor n(1-\delta )\right\rfloor }\frac{1}{d(n-r)}\int_{0\leq
z\,<x\wedge (y+\Delta )}G(r,x-dz)\int_{(y-z)^{+}}^{y+\Delta -z}V(u)du  \notag
\\
&\leq &\frac{CV_{\Delta }(y)}{d(n)}\sum_{\left\lfloor n\delta \right\rfloor
<r<\left\lfloor n(1-\delta )\right\rfloor }\int_{0\leq z\,<x}G(r,x-dz) 
\notag \\
&\leq &\frac{C}{d(n)}\sum_{\left\lfloor n\delta \right\rfloor
<r<\left\lfloor n(1-\delta )\right\rfloor }\sum_{m=0}^{\left\lfloor
x\right\rfloor }G(r,[m,m+1)  \notag \\
&\backsim &\frac{CV_{\Delta }(y)}{d(n)}\sum_{\left\lfloor n\delta
\right\rfloor <r<\left\lfloor n(1-\delta )\right\rfloor
}\sum_{m=0}^{\left\lfloor x\right\rfloor }\frac{\int_{m}^{m+1}U(v)dv}{d(r)} 
\notag \\
&\leq &\frac{Cx}{c_{n}}\cdot \frac{V_{\Delta }(y)U(x+1)}{d(n)}=o(\frac{%
V_{\Delta }(y)U(x)}{d(n)}).  \label{6.7}
\end{eqnarray}

After reading off the asymptotic behaviour of the first term in (\ref{6.11})
from (\ref{6.2}), the proof is now completed by using (\ref{6.5}), (\ref{6.6}%
), (\ref{6.7}), and the following result.
\end{proof}

\begin{lemma}
For $x,y\geq 0$ and $\Delta >0$ the following identity holds%
\begin{eqnarray}
\int_{0\leq z\,<x\wedge (y+\Delta )}\int_{y\vee z\leq w<y+\Delta
}U(x-z)dzV(dw-z) &+&U(x-dz)V(w-z)dw  \notag \\
+\boldsymbol{1}_{\{x\leq y+\Delta \}}\int_{(y-x)^{+}}^{y+\Delta -x}V(w)dw
&=&U(x)V_{\Delta }(y).  \label{6.12}
\end{eqnarray}
\end{lemma}

\begin{proof}
Assume first that $y\geq x,$ so that $x\wedge (y+\Delta )=x,$ and the first
integral reduces to%
\begin{eqnarray*}
&&\int_{0\leq z\,<x}\int_{y\leq w<y+\Delta }U(x-z)dzV(dw-z)+U(x-dz)V(w-z)dw
\\
&=&\int_{0\leq z\,<x}U(x-z)[V((y-z+\Delta )-)-V((y-z)-)]dz+U(x-dz)V_{\Delta
}(y-z) \\
&=&\int_{0\leq z\,<x}-\frac{d}{dz}[U(x-z)V_{\Delta }(y-z)]du=U(x)V_{\Delta
}(y)-U(0)V_{\Delta }(y-x).
\end{eqnarray*}%
This verifies (\ref{6.12}), since $U(0)=1$ and the second term on the LHS of
(\ref{6.12}) is $V_{\Delta }(y-x)$ when $y\geq x.$ If $y<x$ we split the
first integral into two parts and repeat the above calculation to see that 
\begin{eqnarray}
&&\int_{0\leq z\,<y}\int_{y\leq w<y+\Delta }U(x-z)dzV(dw-z)+U(x-dz)V(w-z)dw 
\notag \\
&=&U(x)V_{\Delta }(y)-U(x-y)V_{\Delta }(0).  \label{6.13}
\end{eqnarray}%
The second part is, writing $\overline{V}(z)=\int_{0}^{z}V(w)dw,$ 
\begin{eqnarray}
&&\int_{y\leq z\,<x\wedge (y+\Delta )}\int_{z\leq w<y+\Delta
}U(x-z)dzV(dw-z)+U(x-dz)V(w-z)dw  \notag \\
&=&\int_{y\leq z\,<x\wedge (y+\Delta )}U(x-z)V((y+\Delta -z)-)dz+U(x-dz)%
\overline{V}(y+\Delta -z)  \notag \\
&=&\int_{y\leq z\,<x\wedge (y+\Delta )}-\frac{d}{dz}[U(x-z)\overline{V}%
(y+\Delta -z)]  \notag \\
&=&U(x-y)V_{\Delta }(0)-U(x-(x\wedge (y+\Delta ))\overline{V}(y+\Delta
-(x\wedge (y+\Delta ))  \notag \\
&=&U(x-y)V_{\Delta }(0)-\overline{V}(y+\Delta -x)\boldsymbol{1}_{\{y+\Delta
>x\}}.  \label{6.14}
\end{eqnarray}%
Since the second term in (\ref{6.12}) reduces to $\overline{V}(y+\Delta -x)%
\boldsymbol{1}_{\{y+\Delta >x\}}$ when $y<x,$ the proof in this case follows
from (\ref{6.13}) and (\ref{6.14}).
\end{proof}

\begin{remark}
The recent paper \cite{abkv} contains some functional limit theorems for
conditional random walks in the domain of attraction of a one-sided stable
law.
\end{remark}

\end{document}